\documentclass[12pt]{article}
\usepackage{amsmath}
\usepackage{amsthm,amscd,amsfonts}
\usepackage{amssymb, upref, color}
\usepackage{amsmath,amssymb,amsthm}
\usepackage{color}
\usepackage[colorlinks]{hyperref}
\usepackage{graphicx}
\usepackage{epsf,epsfig,subfigure, verbatim}
\usepackage{latexsym,bm}
\usepackage{enumerate}
\usepackage{listings}
\usepackage[numbers,sort&compress]{natbib}

\usepackage{mathrsfs}
\usepackage{color}
\def \red {\textcolor{red} }

\numberwithin{equation}{section}

\theoremstyle{plain}
\newtheorem{exam}{Example}[section]
\newtheorem{theorem}[exam]{Theorem}
\newtheorem{lemma}[exam]{Lemma}
\newtheorem{remark}[exam]{Remark}

\newtheorem{definition}[exam]{Definition}
\newtheorem{corollary}[exam]{Corollary}

\newtheorem{example}{Example}[section]
\begin{document}

\title{The eigenvector-eigenvalue identity for the quaternion matrix with its algorithm and computer program
	\footnote{ This paper was jointly supported from the National Natural
		Science Foundation of China under Grant (No.11671176, 11931016).}
}
\author
{
	Yuchao He\,\,\,\,\,Mengda Wu\,\,\,\,
	Yonghui Xia$\footnote{Corresponding author. Yonghui Xia, yhxia@zjnu.cn;xiadoc@163.com.}$
	\\
	{\small \textit{$^a$ College of Mathematics and Computer Science,  Zhejiang Normal University, 321004, Jinhua, China}}\\
	{\small Email:  YuchaoHe@zjnu.edu.cn; medawu@zjnu.edu.cn;  yhxia@zjnu.cn;xiadoc@163.com.}
}

\maketitle

\begin{abstract}
	Peter Denton, Stephen Parke, Terence Tao and Xining Zhang \cite{D-P-T-Z} presented a basic and important identity in linear commutative algebra, so-called {\bf the eigenvector-eigenvalue identity} (formally named in \cite{tao-eig}), which is a convenient and powerful tool to succinctly determine eigenvectors from eigenvalues. The identity relates the eigenvector component to the eigenvalues
	of $A$ and the minor $M_j$, which is formulated in an elegant form as follows
	\[
	\lvert v_{i,j} \rvert^2\prod _{k=1;k\ne i}^{n-1}({\lambda_i}(A)-{\lambda_k}(A))=\prod _{k=1}^{n-1}({\lambda_i}(A)-{\lambda_k}(M_j)). \,\,\,
	\]
	In this paper, we extend the eigenvector-eigenvalue identity to the quaternion division ring, which is non-commutative. A version of eigenvector-eigenvalue identity for the quaternion matrix is established. Furthermore, we give a new method and algorithm to compute the eigenvectors from the right eigenvalues for the  quaternion Hermitian matrix. A program is designed to realize the algorithm to compute the eigenvectors. An open problem ends the paper. Some examples show a good performance of the algorithm and the program. 
	

\end{abstract}
\section{Introduction}
\subsection{Eigenvector-eigenvalue identity}

Recently, Peter Denton, Stephen Parke, Terence Tao and Xining Zhang \cite{D-P-T-Z} presented {\bf the eigenvector-eigenvalue identity} which is very convenient for us to determine eigenvectors from eigenvalues. The identity is succinctly formulated as
\begin{equation}\label{EEI}
	\lvert v_{i,j} \rvert^2\prod _{k=1;k\ne i}^{n-1}({\lambda_i}(A)-{\lambda_k}(A))=\prod _{k=1}^{n-1}({\lambda_i}(A)-{\lambda_k}(M_j)),
\end{equation}
where $\lambda_1,\lambda_2,\cdots,\lambda_n$ are the $n$ real eigenvalues of the $n\times n$ Hermitian matrix $A$, $v_1,v_2,\cdots,v_n$ are their corresponding standard orthogonal eigenvectors; $v_{ij}$ denotes the $j$-th element of $v_i$;  $M$ is a $(n-1)\times (n-1)$ submatrix by removing $j$-th row and $j$-th column of $A$.
The identity tells us that the $j$-th component $v_{i,j}$ of a unit eigenvector $v_i$
associated to the eigenvalue $\lambda_i(A)$ is related to the eigenvalues $\lambda_1(M_j)\cdots \lambda_n(M_j)$ of the minor $M_j$ of $A$ formed by removing the $j$-th row and column.

The elegant identity 
establishes the direct relationship between the eigenvector component with the eigenvalues
of $A$ and $M_j$.
In fact,
the identity has not had a standard (unified) name or a standard form before it was formally named by \cite{tao-eig}.
Different versions of such identities have been developed in different names and different forms, {and it has been widely applied in graph theory, inverse eigenvalue problem, random matrix theory, numerical linear algebra,and even made great contributions to neutrino physics (see \cite{toshey}).} Because it has no standard name or a standard form, it is difficult to search for occurrences of the identity until the most recent rediscovery \cite{D-P-T-Z,D-P-X}. To avoid the independent rediscoveries of the identity, 
\cite{tao-eig} conducted a nice and systematic paper on  all the appearances of the identity. And in their paper, they formally named the identity as {\bf the eigenvector-eigenvalue identity} and provided several elegant proofs based on the different methods including the adjugate proof, the Cramer’s rule proof, the Coordinate-free proof, the proof using perturbative analysis, the proof using a Cauchy-Binet type formula and the proof using an alternate expression for eigenvector component magnitudes. Also they provided the generalizations of the eigenvector-eigenvalue identity,  and discussed
applications of the identity.

A very complete and detailed history of the eigenvector-eigenvalue identity \eqref{EEI} was provided in the paper \cite{tao-eig}. The authors have spent a lot of time reading and reviewing all the existing works related to the identity \eqref{EEI}. They  provided a citation graph (see Figure 1 in \cite{tao-eig}) to reveal the citation relations between the existing papers related to the identity. We will not repeat the details here. But with respect, we should also thank again to the mathematicians who had contributed to the identity (e.g. \cite{tao-eig,tao,D-P-X,D-P-T-Z,peter,jacobi,LOWNER,G-E,demmel,weinberger,thompson,THOMPSON3,thompson2,D-H,silov,PAIGE,PARLETT,g-l,Golub,Gladwell,boley,G-K-V,knyazey,k-s,XU,chu,D-N-T,M-D,l-f,HAGOS,cvetkovi,Godsil1,G-M,Godsil2,GKMG,N-T-U,bebiano,Yu. Baryshnikov,Forrester,Dumitriu,Mieghem,kAUSEL,wolchover,esy,cramer,birkhoff,G-S,x.chen,stawiska,Mieghem1}).

\subsection{Motivation}

It is known that the eigenvector-eigenvalue identity \eqref{EEI} is valid for the commutative algebra.
However, is the eigenvector-eigenvalue identity \eqref{EEI} valid for the noncommutative algebra? We pay particular attention to the quaternion which is a special noncommutative algebra. Quaternion and the quaternion matrix have attracted many scholars' attentions due to its important applications in the description of protein structure, neural networks, quantum mechanics, fluid mechanics, {Frenet
	frame in differential geometry, kinematic modeling, attitude dynamics,
	Kalman filter design, spatial rigid body transformation, and so on \cite{Xia-SPAM,XIA-JMP,Xia-Book,WangC2,WangJR1,WangJR2,WangJR3,AH,CZRP,F-Y,Baker,compute}.}  It is known that quaternion is a generalization of the complex number, but quaternion is a division ring. Due to its non-commutativity, a lot of properties are completely different from the complex. In particular, we should deal with the eigenvalues by distinguishing the right eigenvalues with the left eigenvalues. Further more, it is known that $n$-degree polynomial in the complex field has no more than $n$ roots. But this is not valid for quaternion ring. Eigenvalues of complex matrices
satisfy Brauer's theorem \cite{Br} for the inclusion of the eigenvalues, whereas right eigenvalues of quaternion matrices do not have this property. The number of eigenvalues may not be limited (e.g. see Farid et al \cite{ZhangFZ},	Li et al \cite{compute}, Zhang and Wei \cite{F-Y},  Chen \cite{chen1}, Baker \cite{Baker}, Brenner \cite{Bren}). Therefore, it is unknown that the eigenvector-eigenvalue identity \eqref{EEI} is  still valid for the quaternion ring.

In the present paper, we extend eigenvector-eigenvalue identity \eqref{EEI} to the quaternion division ring. A version of eigenvector-eigenvalue identity for the quaternion matrix is established. Furthermore, due to the non-commutativity of quaternions, the standard adjoint matrix method will lead us to nowhere. By introducing the definition of quaternion adjoint matrix, we give a new method and algorithm to compute the eigenvectors from the right eigenvalues for the  quaternion Hermitian matrix. A program is designed to realize the algorithm to compute the eigenvectors.
Some examples show a good performance of the algorithm and the program.

\subsection{Outline of the paper}

Section 2 is to present the conceptions and notations. We state and prove our main results in Section 3. {Due to the non-commutativity of quaternions, the standard adjoint matrix method will lead us to nowhere. In order to overcome this difficulty, we should introduce the concept of the quaternion adjoint matrix to complete the proof.} Section 4 is
to give the algorithms and program  to compute the right eigenvalues of quaternion Hermitian matrix and the elements of the corresponding eigenvectors.

\section{Preliminaries}

In this section, we briefly introduce some basic properties of quaternion.  We recall some notations and preliminary results from \cite{Xia-Book,chen1}.
We need the following notations. $\mathbb{R}$ and $\mathbb{Q}$ denote the real field and the quaternion division ring respectively. $\mathbb{F}^{m \times n}$ denotes the
set of all $m \times n$ matrices on $\mathbb{F}$. For any $A \in \mathbb{Q}^{n \times n}$, $A^*$ refer to the conjugate transpose
of $A$.
A quaternion $a \in \mathbb{Q}$ is defined by
$$a = a_0 + a_1\mathbf{i} + a_2\mathbf{j} + a_3\mathbf{k},$$
where $a_0, a_1, a_2, a_3\in \mathbb{R}$. The imaginary units $\mathbf{i},\mathbf{j},\mathbf{k}$ satisfy\\
$$\mathbf{i}^2=\mathbf{j}^2=\mathbf{k}^2=\mathbf{ijk}=-1,$$
which implies $\mathbf{jk}=-\mathbf{kj}=\mathbf{i}, \mathbf{ki}=-\mathbf{ik}=\mathbf{j}, \mathbf{ij}=-\mathbf{ji}=\mathbf{k}.$ The conjugate of $a$ $\in \mathbb{Q}$ is given by $\bar{a}= a_0-a_1\mathbf{i}-a_2\mathbf{j}-a_3\mathbf{k}$ and the
module $|a|$ is defined by
$$|a|=a_0^2+a_1^2+a_2^2+a_3^2=\bar{a}a=a\bar{a}.$$
The determinant is defined based on permutation  (see \cite{Xia-Book,XIA-JMP,chen1}).
\begin{definition}\label{det}
	For any $A \in \mathbb{Q}^{n \times n},$
	$$\det A=\det \begin{pmatrix}a_{11} &a_{12}&\cdots &a_{1n}\\a_{21} &a_{22}&\cdots &a_{2n}\\
		\vdots&\vdots&\ddots\ & \vdots\\
		a_{n1} &a_{n2}&\cdots &a_{nn} \end{pmatrix}$$
	:=$\Sigma_{\sigma\in S_n}\epsilon(\sigma)a_{n_1i_2}a_{i_2i_3}\cdots
	a_{i_{s-1}i_s}a_{i_sn_1}a_{n_2j_2}\cdots a_{j_tn_2}\cdots a_{n_rk_2}
	\cdots a_{k_ln_r},$
	where $S_n$ is the symmetric group on $n$ letters,and the disjoint cycle decomposition of $\sigma \in S_n$ is written in the normal form:
	\begin{equation*}
		\begin{array}{l}
			\sigma=(n_1i_2i_3\cdots i_s)
			(n_2j_2j_3\cdots j_t)\cdots
			(n_rk_2k_3\cdots k_l),\\
			n_1>i_2,i_3,\cdots,i_s,
			n_2>j_2,j_3,\cdots,j_t,\cdots,
			n_r>k_2,k_3,\cdots,k_l,\\
			n=n_1>n_2>\cdots>n_r\geq1,
		\end{array}
	\end{equation*}
	and
	$$\epsilon(\sigma)=(-1)^{(s-1)+(t-1)+\cdots+(l-1)}=(-1)^{n-r}.$$
\end{definition}
Under the above definition, the quaternion Hermitian matrix $A$ (i.e. $A=A^*\in \mathbb{Q}^{n\times n}$) has following properties \cite{row,chen1} (In Lemma \ref{diag1}-\ref{change-det}, we always assume that $A$ is a quaternion Hermitian matrix).
\begin{lemma}\label{diag1}
	The matrix $\lambda E-A$ can be transformed into a centralized matrix in the form of $\mathrm{diag}(1,1,\cdots,1,\phi_1(\lambda),\phi_2(\lambda),\cdots,\phi_s(\lambda))$ by elementary transformation,
	where $\phi_i(\lambda)|\phi_{i+1}(\lambda),$ $ i=1,2,\cdots,s-1.$ Moreover, we have $$ \det(A)=(-1)^n\phi_1(0)\phi_2(0)\cdots\phi_n(0).$$
\end{lemma}
\begin{lemma}\label{right-eig} $A$ has $n$ right eigenvalues, which are all real numbers.
\end{lemma}
\begin{lemma}\label{right-diag}
	There exists a unitary matrix $V \in \mathbb{Q}^{n \times n}$, such that
	$$V^*AV=\mathrm{diag} \{\lambda_1,\lambda_2,\cdots,\lambda_n\},$$
	where $\lambda_1,\lambda_2,\cdots,\lambda_n$ are the right eigenvalues of $A$.\\
\end{lemma}
\begin{lemma}\label{det-eig}If $\lambda_1(A),\lambda_2(A),\cdots,\lambda_n(A)$ are the right eigenvalues of $A$, then $$\det(A)=\prod _{i=1}^{n}{\lambda_i}(A).$$
\end{lemma}
\begin{lemma}\label{change-det}(see \cite{chen1}) Let $P(k,j_\lambda)$ be the matrix obtained by adding $\lambda (\lambda\in \mathbb{Q})$ times the $j$-th column to the $k$-th column of unit $n \times n$ matrix $E$. If $H^*=H\in \mathbb{Q}^{n\times n}$, then
	$$\det P^*(k,j_\lambda)HP(k,j_\lambda)=\det H,\quad
	j\ne k=1, 2, \cdots, n.$$
\end{lemma}

\section{Main results and the proofs}
\subsection{Caculations of the modulus of the eigenvector elements}
In this section, we present the main results of this paper. Firstly we extend Cauchy-Binet type formula to quaternions:
\begin{theorem}
	\label{cauchy}[Quaternionic Cauchy-Binet type formula]
	Suppose that $A^*=A \in \mathbb{Q}^{n \times n}$ and one of the right eigenvalues is zero, $v$ is the corresponding eigenvector to {zero eigenvalue}. Then,
	\begin{equation}\prod _{i=1}^{n-1}{\lambda_i}(A)\det((B\quad v)^*(B\quad v))=\det(B^*AB),\end{equation}
	for any $n \times (n-1)$ quaternion matrix $B$.\end{theorem}
\begin{proof}
	Let $\lambda_1,\lambda_2,\cdots,\lambda_{n}$ be the right eigenvalues of $A$. In view of Lemma \ref{right-eig}, $\lambda_i$ $(i=1,2,3,\cdots,n-1)$ are all real numbers. Without loss of generality, we set \label{e1}$\lambda_n(A)=0$. There exists a unitary matrix $V\in \mathbb{Q}^{n\times n}$, such that
	$ V^*AV=D $, where $D=\mathrm{diag}(\lambda_1,\lambda_2,\cdots,\lambda_{n-1},0)$.
	Set $$e_n=(0,0,\cdots,0,1)^T,\quad P=V^*B=\begin{pmatrix}B_1\\X\end{pmatrix},$$where $B_1$ is the upper $(n-1) \times (n-1)$ submatrix and $X$ is some $1\times(n-1)$ vector. We see that Eq. $(\ref{e1})$ holds if and only if the identity $$\det(P^*DP)=\prod _{i=1}^{n-1}{\lambda_i}\det((VP\quad Ve_n)^*(VP\quad Ve_n))$$holds. Note that
	$$\det(P^*DP) =\det\left( \begin{pmatrix}B_1^*&X^*\end{pmatrix}
	\mathrm{diag}(\lambda_1,\lambda_2,\cdots,\lambda_{n-1},0)\begin{pmatrix} B_1\\X \end{pmatrix}\right)=\prod _{i=1}^{n-1}{\lambda_i}\det(B_1^*B_1).$$
	Then we have
	\[\begin{split} &\prod _{i=1}^{n-1}{\lambda_i}\det((VP\quad Ve_n)^*(VP\quad Ve_n))\\
		=&\prod _{i=1}^{n-1}{\lambda_i}\det \begin{pmatrix} {B_1^*B_1+X^*X} & X^* \\ X & 1 \end{pmatrix} =\prod _{i=1}^{n-1}{\lambda_i}\det \begin{pmatrix} {B_1^*B_1} & 0 \\ 0 & 1 \end{pmatrix} =\prod _{i=1}^{n-1}{\lambda_i}\det(B_1^*B_1)=\det(P^*DP).
	\end{split}\] The proof is completed.
\end{proof}
Now we are in a position to state the eigenvector-eigenvalue identity for quaternion matrix.
\begin{theorem}
	\label{result}
	[eigenvector-eigenvalue identity for quaternion matrix] Suppose that $A^*=A \in \mathbb{Q}^{n\times n}$, $v_{ij}$ is the $j$-th element of eigenvector corresponding to the eigenvalue $\lambda_i$. $M_j$ is the $(n-1) \times (n-1)$
	submatrix of $A$ that results from deleting the $j$-th column and the $j$-th row, with eigenvalues $\lambda_k(M_j)$. Then
	\begin{equation}\label{e2}\lvert v_{i,j} \rvert^2\prod _{k=1;k\ne i}^{n-1}({\lambda_i}(A)-{\lambda_k}(A))=\prod _{k=1}^{n-1}({\lambda_i}(A)-{\lambda_k}(M_j)).\end{equation}
\end{theorem}
\begin{proof}Without loss of generality, we take $j=1,i=n$. We shift $A$ by $\lambda_n(A)E_n$ so that $\lambda_n(A)=0.$ This also shifts all the remaining eigenvalues of $A$ as well as those of $M_j$, then Eq. (\ref{result}) is equivalent to
	\begin{equation}\label{e3} \lvert v_{n,1} \rvert^2\prod _{k=1}^{n-1}{\lambda_k}(A)=\prod _{k=1}^{n-1}{\lambda_k}(M_1).
	\end{equation}
	Let $B=\begin{pmatrix}0\\E_{n-1} \end{pmatrix}$. By Eq. ({\ref{cauchy}}) and (\ref{e3}), we  get
	\begin{equation*}
		\begin{split}
			\lvert v_{n,1} \rvert^2\prod _{k=1}^{n-1}{\lambda_k}(A)&=\prod _{k=1}^{n-1}{\lambda_k}(A)\det((B,v_n)^*(B,v_n))\\&=\det(B^*AB)=\det(M_1)=\prod _{k=1}^{n-1}{\lambda_k}(M_1).
		\end{split}
	\end{equation*}The proof is completed.
\end{proof}
\begin{corollary}\label{submatrix}If $A=A^*\in \mathbb{Q}^{n \times n}$, the eigenvector element $v_{ij}$ is $0$, then $M_j$ has at least one right eigenvalue equal to $\lambda_i$.
\end{corollary}
\begin{corollary}\label{left-eig}If $A=A^*\in \mathbb{Q}^{n \times n}$ and the left eigenvalue $\hat{\lambda}$ is real number, then it's equal to some right eigenvalue $\lambda_i$ and
	$$ \lvert \hat{v}_j \rvert^2\prod _{k=1;k\ne i}^{n-1}(\hat{\lambda}(A)-{\lambda_k}(A))=\prod _{k=1}^{n-1}(\hat{\lambda}(A)-{\lambda_k}(M_j)).$$
\end{corollary}
In fact, the right eigenvalues are real numbers when $A=A^*$ and
$$Av=v\lambda =\lambda v.$$
It means that the left eigenvalues of real numbers exist and the right eigenvalues belong to the set of left eigenvalues.

\begin{remark}
	The formula in Theorem \ref{result} provides a new method different from the traditional method   to calculate the element modules of the right eigenvalue of the quaternion matrix corresponding to the unit eigenvector. A simple example is given to verify the correctness of Theorem \ref{result}.
\end{remark}

\begin{example}
	\label{model}Calculate the element modules of the right eigenvalue of the quaternion matrix $A=\begin{pmatrix}3&\mathbf{i}-\mathbf{j}+\mathbf{k}\\
		-\mathbf{i}+\mathbf{j}-\mathbf{k}&2 \end{pmatrix}.$
\end{example}
\noindent {\bf Traditional method}:  Direct computation gives
$$\det(A)=3,
\lambda_1=\frac{5}{2}+\frac{\sqrt{13}}{2}, \lambda_2=\frac{5}{2}-\frac{\sqrt{13}}{2}.
$$
The unit right eigenvector corresponding to $\lambda_1$ is $$ v=\begin{pmatrix}\displaystyle\frac{2\sqrt6(\mathbf{i}-\mathbf{j}+\mathbf{k})}{(\sqrt{13}-1)\sqrt{13+\sqrt{13}}}\\
	\dfrac{\sqrt6}{\sqrt{13+\sqrt{13}}}
\end{pmatrix}.$$
Then
$$\lvert v_1 \rvert^2=\frac{\sqrt{13}+13}{26},
\quad\lvert v_2 \rvert^2= \frac{13-\sqrt{13}}{26}.
$$
\noindent {\bf Method of by using eigenvector-eigenvalue identity}: 	Using the result of Theorem \ref{result}, we get
$$\lvert v_1 \rvert^2=\frac{\lambda_1-\lambda_2}{\lambda_1-\lambda(M_1)}
=\frac{\sqrt{13}+13}{26},
\quad\lvert v_2 \rvert^2=\frac{\lambda_1-\lambda_2}{\lambda_1-\lambda(M_2)}
=\frac{13-\sqrt{13}}{26},
$$
which is consistent with the result obtained by traditional calculation. However, by using eigenvector-eigenvalue identity, it is more convenient for us  to calculate the element modules of the right eigenvector of the quaternion matrix.


We remark that Theorem \ref{result} is to compute the element modules of the right eigenvalue. In order to compute the right eigenvectors for quaternion matrix, we provide a new method based on quaternion adjoint matrix in next subsection.

\subsection{Caculations of eigenvectors based on quaternion adjoint matrix }

To calculate the eigenvectors of quaternion matrices, we first introduce  the concepts of row expansion \cite{row} and adjoint of quaternion matrices \cite{qadj}.\\
\begin{definition}\label{row expansion}
	$A=(a_{ij})$ is any quaternion matrix of order $n$.
	$$|A|^{row}=\sum{\pm a_{1p}a_{pq}\cdots a_{st}}$$
	is the row expansion of $A$, where each term corresponds to the following permutation group:
	\begin{equation}\label{group}
		\begin{pmatrix}
			1&p&q&\cdots&s\\
			p&q&r&\cdots&t
		\end{pmatrix}.
	\end{equation}
	The permutation group $(\ref{group})$ can be decomposed into the product of $d$ independent rotations, and the positive and negative of each item is the same as $(-1)^{n-d}$. This is the same as the usual sign rule for determinant expansions. However, the factor order of each item should be arranged according to the principle that the first factor can be selected in the first row. If it is $a_{1p}$ and $p\ne 1$, the second factor can be selected in the $p$ row; If the first factor is $a_{11}$, the second factor must be in the second row; If the first factor is $a_{1p}$ and the second factor is $a_{p1}$, the third must be the row with the smallest number of rows among the remaining rows except the first row and the $p-th$ row. The other factor selection methods follow this example.
\end{definition}

\begin{example}
	$\left|\begin{array}{cccc}
		a_{11}&a_{12}&a_{13}&a_{14}\\
		a_{21}&a_{22}&a_{23}&a_{24}\\
		a_{31}&a_{32}&a_{33}&a_{34}\\
		a_{41}&a_{42}&a_{43}&a_{44}
	\end{array}\right|^{row}\!\!\!\!\!\!\!\!\!=a_{11}a_{22}a_{33}a_{44}+a_{11}a_{23}a_{34}a_{42}+a_{11}a_{24}a_{43}a_{32}-a_{11}a_{24}a_{42}a_{33}\\
	-\cdots +a_{12}a_{21}a_{34}a_{43}+ \cdots +a_{14}a_{41}a_{23}a_{32}.
	$
\end{example}
\begin{definition}\label{qadj}
	$\mathrm{qadj}(A)$ is the adjoint matrix of $n$-order quaternion Hermitian matrix, which is defined by
	\begin{equation}\label{specific}
		\mathrm{qadj}(A)=\begin{pmatrix}
			\quad|A_{11}|^{row}&-|A_{21}|^{row}&-|A_{31}|^{row}&\cdots&-|A_{n1}|^{row}\\
			-|A_{12}|^{row}&\quad|A_{22}|^{row}&-|A_{32}|^{row}&\cdots&-|A_{n2}|^{row}\\
			-|A_{13}|^{row}&-|A_{23}|^{row}&\quad|A_{33}|^{row}&\cdots&-|A_{n3}|^{row}\\
			\vdots&\vdots&\vdots& &\vdots\\
			-|A_{1n}|^{row}&-|A_{2n}|^{row}&-|A_{3n}|^{row}&\cdots&\quad|A_{nn}|^{row}
		\end{pmatrix},
	\end{equation}
	where $A_{ij}$ is the natural submatrix (see \cite{row}), i.e., remove row $i$ and column $j$ of the matrix $A$ and rearrange $j-2$, $1$, $2$, $\cdots$, $j-1$ rows  into $1$, $2$, $\cdots$, $j-2$, $j-1$ rows, and the remaining $n-j$ rows remain in their original positions.
\end{definition}
\begin{lemma}\label{qadj2}(see \cite{qadj})
	If $A\in \mathbb{Q}^{n\times n}$ is a quaternion Hermitian matrix, then
	\begin{equation}\label{qadj_i}
		\mathrm{qadj}(A)A=A\mathrm{qadj}(A)=\det(A)E.
	\end{equation}
\end{lemma}
Based on the above definitions and theorems, we now give a new method to calculate the quaternion eigenvector. Even the set of eigenvectors can be standard orthogonal.
\begin{theorem}\label{eigenvector}
	$A=A^*\in \mathbb{Q}^{n \times n}$, $\lambda_1, \lambda_2, \cdots, \lambda_n$ is the right eigenvalues of $A$ and $v_1, v_2, \cdots, v_n$ are the corresponding standard orthogonal eigenvectors, then
	\begin{equation}\label{calculate}
		\mathrm{qadj}(\lambda_i(A)E_n-A)=\prod_{k=1;k\ne i}^n(\lambda_i(A)-\lambda_k(A))v_iv_i^*.
	\end{equation}
\end{theorem}
\begin{proof}
	According to (\ref{qadj_i}), if $\lambda$ is not the eigenvalue of $A$, we have
	$$\mathrm{qadj}(\lambda E_n-A)=\det(\lambda E_n-A)(\lambda E_n-A)^{-1},$$
	which leads to
	\[\begin{split}
		\mathrm{qadj}(\lambda E_n-A)v_j&=\det(\lambda E_n-A)(\lambda-\lambda_j(A)^{-1})v_j
		\\&=\prod_{k=1;k\ne j}(\lambda-\lambda_k(A))v_j,\quad j=1, 2, \cdots, n.
	\end{split}\]
	Then \begin{equation}\label{111}\mathrm{qadj}(\lambda E_n-A)=\sum_{j=1}^n\prod_{k=1;k\ne j}^n(\lambda-\lambda_k(A))v_jv_j^*.\end{equation}\par
	According to the Definition \ref{row expansion}, \ref{qadj}, we can see that every element of matrix $\mathrm{qadj}(\lambda(A)E_n-A)$ of order $n$ is a polynomial with respect to $\lambda$, which means  $\mathrm{qadj}(\lambda(A)E_n-A)$ is a continuous function with respect to $\lambda$. By taking the limit $\lambda\to \lambda_i(A)$ on both sides of  equation (\ref{111}), we get
	$$\mathrm{qadj}(\lambda_i(A)E_n-A)=\prod_{k=1,k\ne i}^n(\lambda_i(A)-\lambda_k(A))v_iv_i^*$$
\end{proof}
If $A$ is a $n\times n$ Hermitian matrix over complex field, then $\mathrm{qadj}(A)$ is reduced to $\mathrm{adj}(A)$. Then we can get the following identity proposed by Peter B. Denton, Stephen J. Parke, Terence Tao, Xining Zhang \cite{tao-eig}.
\begin{corollary}\label{tao}
	If $A$ is a $n\times n$ Hermitian matrix over complex field,
	$$	\mathrm{adj}(\lambda_i(A)E_n-A)=\prod_{k=1;k\ne i}^n(\lambda_i(A)-\lambda_k(A))v_iv_i^*.$$
\end{corollary}
Through equation (\ref{calculate}), the corresponding elements of row $p$ and column $q$ of the left and right matrices are equal, we can get
\begin{equation}\label{matrix equal}
	\left\{\begin{array}{rl}
		|(\lambda_i(A)E_n-A)_{pp}|^{row}&=\prod_{k=1;k\ne i}^n(\lambda_i(A)-\lambda_k(A))|v_{ip}|^2,\\
		-|(\lambda_i(A)E_n-A)_{qp}|^{row}&=\prod_{k=1;k\ne i}^n(\lambda_i(A)-\lambda_k(A))v_{ip}v_{iq},\quad p\ne q.
	\end{array}\right.
\end{equation}

If $\lambda_i(A)$ is not equal to the other eigenvalues of $A$, we can calculate the normalized eigenvector corresponding to $\lambda_i(A)$ through the formula.
Noticing that one of elements of the eigenvectors is not zero at least, we assume $v_{im}\ne 0$.
Since the right eigenvalue of the Hermitian matrix is a real number, it is still an eigenvector to multiply the right of the eigenvector by a quaternion:
$$Av_i=v_i\lambda_i \iff Av_iq=v_iq\lambda_i, \quad q\in \mathbb{Q},$$
then we can use this method to control $v_{im}$ to be a real number and calculate the other elements in the vector $v_i$ in combination with the formula (\ref{matrix equal}):\\
\begin{equation}\label{matrix equal2}
	v_{ij}=\left\{\begin{array}{ll}
		\displaystyle
		\sqrt{\frac{|(\lambda_i(A)E_n-A)_{mm}|^{row}}{\prod_{k=1;k\ne i}^n(\lambda_i(A)-\lambda_k(A))}},& j=m,\\&\\\displaystyle
		\frac{-v_{im}^{-1}\overline{|(\lambda_i(A)E_n-A)_{jm}|^{row}}}{\prod_{k=1; k\ne i}^n(\lambda_i(A)-\lambda_k(A))},& j\ne m.
	\end{array}\right.
\end{equation}
We  illustrate  the effectiveness of \eqref{calculate} and \eqref{matrix equal2} through the following example.
\begin{example}\label{calculate2}
	$A=\begin{pmatrix}3&\mathbf{i}-\mathbf{j}+\mathbf{k}\\
		-\mathbf{i}+\mathbf{j}-\mathbf{k}&2 \end{pmatrix}.$
\end{example}
In  Example \ref{model}, we already know that $$\lambda_1=\frac{5}{2}+\frac{\sqrt{13}}{2},\lambda_2=\frac{5}{2}-\frac{\sqrt{13}}{2}.$$
By using Eq. $(\ref{matrix equal2})$, we get $$v_{12}=\sqrt{\frac{|(\lambda_1(A)E-A)_{22}|^{row}}{\lambda_1(A)-\lambda_2(A)}}=\sqrt{\frac{13-\sqrt{13}}{26}},$$ $$v_{11}=\frac{-v_{12}^{-1}\overline{|(\lambda_1(A)E-A)_{12}|^{row}}}{\lambda_1(A)-\lambda_2(A)}=\sqrt{\frac{2}{13-\sqrt{13}}}(\mathbf{i}-\mathbf{j}+\mathbf{k}).$$
We can verify that $v_1=(v_{11},v_{12})^T$ satisfies the equations $Av_1=v_1\lambda_1$ and $|v_{11}|^2+|v_{12}|^2=1$ by direct calculation, which shows that $v_1$ is a unit eigenvector corresponding to the eigenvalue $\lambda_1$.

In fact, we give an algorithm and program to compute the right eigenvectors of the quaternion matrix. By using our program in next section, we see that the program result verifies the obtained identities \eqref{calculate} and \eqref{matrix equal2}. The following list is the input data of Example \ref{calculate2}, and output result by running the designed program in Section 4.

\begin{lstlisting}
	Enter four parts of quaternion matrix B=qadj(a*I_n-A) to be studied:
	B0=[1/2+sqrt(13)/2,0;0,-1/2+sqrt(13)/2];B1=[0,1;-1,0];
	B2=[0,-1;1,0];B3=[0,1;-1,0];
	The value of v11=0.4614i-0.4614j+0.4614k;
	The value of v12=0.6011.	
\end{lstlisting}
\section{Algorithms and programs}
In this section, we introduce a simple method to calculate the right
eigenvalues and the elements of eigenvectors of Hermitian matrices, which is ultimately implemented by MATLAB programs.\\
For any $n\times n$ quaternion Hermitian matrix $A$, it can be uniquely expressed as $A=A_0+A_1\mathbf{i}+A_2\mathbf{j}+A_3\mathbf{k}$, where $A_0, A_1, A_2, A_3\in \mathbb{R}^{n \times n}$.
\red{In order to find the right eigenvalues of $A$, we introduce a mapping $\Phi$: $\mathbb{Q}^{n \times n}\to \mathbb{R}^{4n \times 4n}$ (see \cite{Q-X}),} $$\Phi(A)=\begin{pmatrix} A_0& A_1& A_2 &A_3\\
	-A_1& A_0&-A_3&A_2\\
	-A_2&A_3&A_0&-A_1\\
	-A_3&-A_2&A_1&A_0\end{pmatrix}.$$
Obviously, $\Phi$ is commutative with matrix multiplication and $\Phi$ is a linear mapping, i.e., for any $m\in \mathbb{Q}$ and $A,B\in \mathbb{Q}^{n \times n}$, we have\begin{enumerate}[(i)]
	\item $\Phi(mA)=m\Phi(A)$;
	\item $\Phi(A+B)=\Phi(A)+\Phi(B)$;
	\item $\Phi(AB)=\Phi(A)\Phi(B)$.
\end{enumerate}

\begin{theorem}\label{eigenvalue calculation}
	The right eigenvalues of Hermitian matrix $A$ are the same as the eigenvalues of $\Phi(A)$, and the multiplicity of each eigenvalue of $\Phi(A)$ is $4$ times of $A$.
\end{theorem}

\begin{proof}
	There is an invertible matrix U, such that $U^{-1}(\lambda E-A)U$ is a diagonal matrix for any $\lambda$. By Theorem \ref{right-diag}, we have
	$$\det(\lambda E-A)=\det(U^{-1}(\lambda E-A)U).$$
	Then the characteristic polynomial of $\Phi(A)$ is
	\[\begin{split}
		\det(\lambda E-\Phi(A))&=\det\Phi (\lambda E-A)
		\\&=\det\Phi(U\mathrm{diag}(\lambda-\lambda_1,\lambda-\lambda_2,\cdots,\lambda-\lambda_n)U^{-1})
		\\&=\det(\Phi (U)\Phi(\mathrm{diag}(\lambda-\lambda_1,\lambda-\lambda_2,\cdots,\lambda-\lambda_n))\Phi(U^{-1}))
		\\&=\det\Phi(U^{-1}U)\det\Phi(\mathrm{diag}(\lambda-\lambda_1,\lambda-\lambda_2,\cdots,\lambda-\lambda_n))
		\\&=(\lambda-\lambda_1)^4(\lambda-\lambda_2)^4\cdots(\lambda-\lambda_n)^4.
	\end{split}\]
	Therefore, the eigenvalues of $A$ and $\Phi(A)$ are the same, and the multiplicity of eigenvalues in $\Phi(A)$ is $4$ times that of $A$.
\end{proof}
Thus, the problem of calculating the right eigenvalues of quaternion matrix $A$ is transformed into the problem of calculating the eigenvalues of real matrix $\Phi(A)$. Now we give a program to calculate the modules of the elements of eigenvalues and eigenvectors when $A$ has no multiple eigenvalues. Further, by inputting matrix $\mathrm{qadj}(\lambda I_n-A)$, we can accurately calculate the eigenvectors. For convenience, we label the eigenvalues of $A$ from small to large: $\lambda_1<\lambda_2<\cdots<\lambda_n.$\\
MATLAB code:
\begin{lstlisting}
	%Please input four parts of quaternion matrix A to be studied:A0,A1,
	%A2,A3.
	%Please input the label of the eigenvalue to be calculated:k.
	[n,m]=size(A0);
	mo=zeros(1,n);
	A=[A0,A1,A2,A3;-A1,A0,-A3,A2;-A2,A3,A0,-A1;-A3,-A2,A1,A0];
	[x,y]=eig(A);
	z=sort(sum(y));
	a=zeros(1,n);
	for i=1:n
	a(i)=z(4*i-3);
	end
	%a is the right eigenvalue vector of matrix A.
	M0=zeros(n-1,n-1);
	for p=1:n
	M0(1:p-1,1:p-1)=A0(1:p-1,1:p-1);
	M0(1:p-1,p:n-1)=A0(1:p-1,p+1:n);
	M0(p:n-1,p:n-1)=A0(p+1:n,p+1:n);
	M0(p:n-1,1:p-1)=A0(p+1:n,1:p-1);
	M1=zeros(n-1,n-1);
	M1(1:p-1,1:p-1)=A1(1:p-1,1:p-1);
	M1(1:p-1,p:n-1)=A1(1:p-1,p+1:n);
	M1(p:n-1,p:n-1)=A1(p+1:n,p+1:n);
	M1(p:n-1,1:p-1)=A1(p+1:n,1:p-1);
	M2=zeros(n-1,n-1);
	M2(1:p-1,1:p-1)=A2(1:p-1,1:p-1);
	M2(1:p-1,p:n-1)=A2(1:p-1,p+1:n);
	M2(p:n-1,p:n-1)=A2(p+1:n,p+1:n);
	M2(p:n-1,1:p-1)=A2(p+1:n,1:p-1);
	M3=zeros(n-1,n-1);
	M3(1:p-1,1:p-1)=A3(1:p-1,1:p-1);
	M3(1:p-1,p:n-1)=A3(1:p-1,p+1:n);
	M3(p:n-1,p:n-1)=A3(p+1:n,p+1:n);
	M3(p:n-1,1:p-1)=A3(p+1:n,1:p-1);
	M=[M0,M1,M2,M3;-M1,M0,-M3,M2;-M2,M3,M0,-M1;-M3,-M2,M1,M0];
	[x1,y1]=eig(M);
	z1=sort(sum(y1));
	a1=zeros(1,n-1);
	for i=1:n-1
	a1(i)=z1(4*i-3);
	end
	a2=zeros(1:n-1);
	a2(1:k-1)=a(1:k-1);
	a2(k:n-1)=a(k+1:n);
	mo(p)=sqrt(prod(a(k)-a1)/prod(a(k)-a2));
	end
	%mo(p) is the module of the p-th element of the eigenvector
	%corresponding to the k-th right eigenvalue.
	%Further,we can calculate the eigenvector passes through qadj(a*I_n-A)
	%Please input four parts of quaternion matrix B=qadj(a*I_n-A) to be
	%studied:B0,B1,B2,B3.
	for i=1:n
	if mo(i)~=0
	m=i;
	end
	end
	v=zeros(4,n);
	v(1,m)=sqrt(B0(m,m)/prod(a(k)-a2));
	for i=1:m-1
	v(1,i)=1/v(1,m)*B0(m,i)/prod(a(k)-a2);
	v(2,i)=-1/v(1,m)*B1(m,i)/prod(a(k)-a2);
	v(3,i)=-1/v(1,m)*B2(m,i)/prod(a(k)-a2);
	v(4,i)=-1/v(1,m)*B3(m,i)/prod(a(k)-a2);
	end
	%Columns 1,2,3 and 4 in row x of matrix V represent the real part,i
	%part,j part and k part of the x-th element of the eigenvector
	%respectively.
	
\end{lstlisting}

\section{Discussion and open problem}

In this paper, we provide the eigenvector-eigenvalue identity for the right eigenvalues. But it is known that the left eigenvalues are quite different from the right ones. {And it is more difficult for us to study the left eigenvalues for quaternion matrix (see \cite{left-eig2,left-eig3})}. The eigenvector-eigenvalue identity for the left eigenvalues remains open. 

\section*{Data Availability Statement}

My manuscript has no associated data.

\section*{Contributions}
We declare that all the authors have same contributions to this paper. All the authors approve the final version.

\section*{Conflict of Interest}
\hskip\parindent
The authors declare that they have no conflict of interest.

\end{document}